\documentclass[letterpaper, 10 pt, conference]{ieeeconf}  

\IEEEoverridecommandlockouts                              

\usepackage{graphics} 
\usepackage{epsfig} 
\usepackage{mathptmx} 
\usepackage{times} 
\usepackage{amsmath} 
\usepackage{amssymb}  
\usepackage{float}
\usepackage{dblfloatfix}
\usepackage{bbm}
\usepackage{cite}

\newtheorem{theorem}{Theorem}

\title{\LARGE \bf
On the Consistency of Maximum Likelihood Estimators for Causal Network Identification}

\author{Xiaotian Xie, Dimitrios Katselis, Carolyn L. Beck and R. Srikant 
\thanks{Xiaotian Xie, Carolyn L. Beck and R. Srikant are with the Coordinated Science Laboratory, University of Illinois at Urbana-Champaign, Urbana, IL 61801, Emails: $\{\text{xx5}|\text{beck3}|\text{rsrikant}\}$@illinois.edu.
      }%
\thanks{Dimitrios Katselis is with the ECE Department, University of Illinois at Urbana-Champaign, Urbana, IL 61801, Email: $\text{katselis}$@illinois.edu.}      
}

\begin{document}

\maketitle
\thispagestyle{empty}
\pagestyle{empty}

\begin{abstract}
We consider the problem of identifying parameters of a particular class of Markov chains, called Bernoulli Autoregressive  (BAR) processes. The structure of any BAR model is encoded by a directed graph. Incoming edges to a node in the graph indicate that the state of the node at a particular time instant is influenced by the states of the corresponding parental nodes in the previous time instant. The associated edge weights determine the corresponding level of influence from each parental node. In the simplest setup, the Bernoulli parameter of a particular node's state variable is a convex combination of the parental node states in the previous time instant and an additional Bernoulli noise random variable. 
This paper focuses on the problem of edge weight identification using Maximum Likelihood (ML) estimation and proves that the ML estimator is strongly consistent for two variants of the BAR model. 
We additionally derive closed-form estimators for the aforementioned two variants and prove their strong consistency. 
\end{abstract}

\section{INTRODUCTION}

The spreading of ideas and information, the propagation of viruses and diseases, and the fluctuation of stock prices are examples of processes evolving over social, information or other types of networks \cite{barrat2008,Neely2010,AcemogluDLO2011,JadbabaieMST2012,daneshmand2014estimating,quinn2015directed, pouget2015inferring,NowzariPP2016}. Identifying the underlying network structure in these systems motivates the so-called \textit{network inference problem}, which aims at recovering the underlying connectivity between entities or nodes in the system based on observed data. The dependencies, correlations or causal relationships between network entities can be modeled as undirected or directed edges in a graph. The associated dependency strengths can be described as edge weights. 
Many algorithms have been proposed to identify the network structure and edge weights from time series data for various processes. Clearly, efficient algorithms in terms of sample complexity are desired. 


The network inference problem for various dynamic processes has been recently studied in both the machine learning literature and the system identification literature. Among the relevant studies in the machine learning community, the so-called continuous-time independent cascade model (CICE) considered in \cite{gomez2012inferring, abrahao2013trace, daneshmand2014estimating}, presents a typical model for capturing the dynamics of virus or information spreading in networks. A discrete-time version of CICE is studied in \cite{netrapalli2012learning}. In \cite{pouget2015inferring}, the Generalized Linear Model is formulated, which is a class of diffusion models encompassing both the discrete and continuous-time CICE models, and the Linear Voter model.

System identification is a well-studied model estimation approach in the context of control theory; e.g., see \cite{SoderstromStoica89,ljung1998system,VerVer2007} and references therein. The traditional task of system identification focuses on  estimating system parameters from measured input and output data. 
Among recent emerging applications in the system identification area, two main extensions are relevant. The first involves cases where the state matrix of a state-space model represents a directed graph, in other words, the dynamic process can be factorized into subprocesses for each node \cite{materassi2010, materassi2013model,chiuso2012bayesian, seneviratne2012topology,VandenHofDHB2013}. In such cases, system identification methods are found effective in tackling dynamic network inference problems. The second extension is to consider state vectors evolving in discrete state spaces. 

In this paper, we consider the Bernoulli Autoregressive (BAR) model, which is a parameterized discrete-time Markov chain initially introduced in \cite{katselis2018mixing}. In this model, the state of each node is a Bernoulli random variable with probability of success equal to a convex combination of the parental node states (or their flipped states) in the previous time step and an additional binary noise term ensuring persistence of excitation. The BAR model can be used to approximate opinion dynamics, biological and financial times series, and similar processes \cite{katselis2018mixing}. Another relevant discrete-time binary process is the ALARM model proposed in \cite{agaskar2013alarm}. In contrast to the BAR model, the ALARM model defines the transition probabilities via a logistic function. 

Relying on well-established statistical principles, we first formulate and study the consistency properties of the Maximum Likelihood (ML) parameter estimator for the BAR model in which every parental node causally influences each descendant node positively; the notion of positive correlations is formalized in \cite{katselis2018mixing}. The consistency of ML estimators in the case of independent and identically distributed (i.i.d.) random variables has been studied extensively; see e.g., \cite{moulin2018,poor2013introduction} and references therein. The consistency of ML estimators for Markov chains appears to be less well studied, see \cite{ranneby1978necessary} for a reference.

To establish the (strong) consistency of the ML estimator for the BAR model, we prove that the vectorized transition probability matrix is an injective mapping of the model parameters. In the rest of the paper, we call the injectivity of this mapping \emph{identifiability} of the BAR model.  The strong consistency of the ML estimator is then shown by leveraging the injectivity and the continuity of the transition probabilities with respect to the parameters, as well as the compactness of the parameter space. By relying on the ML principle, a closed-form estimator is subsequently provided. Strong consistency is also shown to hold for this estimator.  The identifiability  proof is then extended to the \emph{generic} BAR model with both positive and negative correlations, where the notion of negative correlations is also formalized in \cite{katselis2018mixing}. This identifiability extension establishes the strong consistency of the ML estimator for the general BAR model class. The closed-form estimator and its consistency are also extended to the \emph{generic} BAR model. These analytical results provide a complement to the prior work \cite{katselis2018mixing}.  Finally, numerical simulations are provided to demonstrate the sample complexity gain achieved by the derived estimators in this paper over other existing algorithms in the literature for the BAR model when focusing on the structure identification subproblem. We note here that structural inference is an identification subproblem that can be tackled by parameter estimation in processes with underlying network structures.         

The rest of this paper is organized as follows. In Section \ref{sec:BAR Model}, the BAR model with positive correlations only is introduced. In Sections \ref{sec:ML Estimation} and \ref{sec:closed form}, the identifiability of the BAR model with positive correlations only and the strong consistency of the corresponding ML estimator, as well as a closed-form estimator and its strong consistency are derived. The generic BAR model with positive and negative correlations and also, the identifiability and strong consistency of the corresponding ML estimator are provided in Section \ref{sec: General BAR}. The extension of the closed-form estimator to the generic BAR model and its strong consistency are derived in Section \ref{sec: closed form 2}. Finally, simulation results are provided in Section \ref{sec:sims} and  Section \ref{sec:concl} concludes the paper. 

\textbf{Notation}: Matrices and vectors are denoted by bold upper and lowercase letters, respectively. Probability distributions in vector form may be either denoted by bold upper or lowercase letters. Random vectors are also denoted by uppercase bold letters, while their corresponding realizations are denoted by lowercase bold letters. Scalar random variables are denoted by uppercase letters. The $i$-th entry of a vector $\mathbf{x}$ is denoted by $x_i$. For a matrix $\mathbf{A}$, $a_{ij}$ corresponds to its $(i,j)$-th entry. Depending on the context, vector and matrix entries may be indexed more generally, e.g., by state elements.  $\mathbf{1}_m$ and $\mathbf{0}_m$ are the $m\times 1$ all-ones and all-zeros vectors, respectively, and $\mathbf{0}_{m\times n}$ is the all-zeros $m\times n$ matrix. $\mathbf{I}_m$ is the $m\times m$ identity matrix. Moreover, $\mathbf{e}_{m,i}$ is the $i$-th column of $\mathbf{I}_m$.
The cardinality of a set $\mathcal{V}$ is denoted by $\left|\mathcal{V}\right|$. For  $m\in\mathbb{N}$, $[m]=\{1,2,\dots, m\}$. Finally, $\mathbb{I}(\cdot)$ stands for the indicator of a set or an event. 
\section{THE BAR MODEL WITH POSITIVE CORRELATIONS}
\label{sec:BAR Model}

The BAR model is a special form of a Markov chain defined on a directed graph $\mathcal{G}=\left(\mathcal{V},\mathcal{E}\right)$ with $\left|\mathcal{V}\right|=p$ nodes. Let $X_i(k)\in \{0,1\}$ be the state of node $i\in [p]$ at time instant $k$ and let $\mathbf{X}(k)\in \{0,1\}^p$ be the associated BAR process state vector at the same time instant. The most natural BAR model, with positive correlations only, is described by
\begin{equation}\label{eq: BAR_def}
    X_i(k+1) \sim  \textnormal{Ber}\left(\mathbf{a}_i^\top\mathbf{X}(k)+b_iW_i(k+1)\right),\ \ \ i=1,\ldots,p,
\end{equation}
where $\mathbf{a}_i\in [0,1]^p, b_i\in [0,1], i=1,\ldots, p$ are parameters of the BAR model and $\textnormal{Ber}(\rho)$ represents the Bernoulli distribution with parameter $\rho$. Additionally, $\{W_i(k+1)\sim\text{Ber}(\rho_{w_i})\}_{i=1}^p$
are independent noise random variables, also independent of $\mathbf{X}(t)$ for any $t<k+1$, where $\rho_{w_i}\in[\rho_{min},\,\rho_{max}]$ for all $i\in[p]$ with $0 <\rho_{min}<\rho_{max}< 1$. Moreover, the initial distribution is  $P_{\mathbf{X}(0)}$, i.e., $\mathbf{X}(0) \sim P_{\mathbf{X}(0)}$. The interpretation here is that the entries of $\mathbf{X}(k+1)$ are conditionally independent Bernoulli random variables given $\mathbf{X}(k)$.

To ensure that the Bernoulli random variables in (\ref{eq: BAR_def}) are well-defined, we require that 
\begin{equation}\label{eq: BAR_const}
    \sum_{j=1}^p a_{ij} + b_i = 1, \quad \forall i\in[p].
\end{equation}
For persistent excitation, we further assume that $b_{i}\geq b_{min},\forall i\in [p]$, where $ b_{min}\in(0,1)$ is a constant. Notice that if $b_i=0$ for all $i\in[p]$, the BAR Markov chain will get absorbed in $\mathbf{0}_p$ or $\mathbf{1}_p$ upon visiting the state $\mathbf{0}_p$ or $\mathbf{1}_p$, respectively.

Furthermore, we assume that $\mathbf{a}_i$ encodes a part of the graph structure through the equivalence 
\begin{equation}\label{eq: edge}
    \left(j,i \right)\in\mathcal{E} \iff a_{ij}>0, \quad\forall i,j \in [p],
\end{equation}
where the ordered pair $(j,i)$ denotes a directed edge from node $j$ to node $i$. The notion of positive correlations in (\ref{eq: BAR_def}) relies on the fact that $a_{ij}>0$  increases the probability of the event $\{X_{i}(k+1)=1\}$ when $X_j(k)=1$. A more general form of the BAR model with both positive and negative correlations is introduced in Section \ref{sec: General BAR}.

We now let $\mathbf{A}=\left[\mathbf{a}_1,\, \mathbf{a}_2,\, \cdots, \mathbf{a}_p \right]^T$, i.e., $\mathbf{a}_r^T$ corresponds to the $r$-th row of $\mathbf{A}$, $\mathbf{b}=\left[b_1,\, b_2,\, \cdots, b_p \right]^T$, $\mathbf{W}=\left[W_1,\, W_2,\, \cdots, W_p \right]^T$ and $\rho_w=\left[\rho_{w_1},\, \rho_{w_2},\, \cdots, \rho_{w_p} \right]^T$. 
We note that $\{\mathbf{X}(k)\}_{k\geq 0}$ is an irreducible and aperiodic Markov chain with finite state space $\{0,1\}^p$. Moreover, for any vectors $\mathbf{u},\mathbf{v}\in \{0,1\}^p$ 
\begin{align}\label{eq: tranProb}
   p_{\mathbf{uv}}&=P\left(\mathbf{X}(k+1)=\mathbf{v}|\mathbf{X}(k)=\mathbf{u}\right)  \nonumber\\
   &= E_{\mathbf{W}}\left[P\left(\mathbf{X}(k+1)=\mathbf{v}|\mathbf{X}(k)=\mathbf{u},\mathbf{W}\right) \right]\nonumber\\
   &=\prod_{i=1}^p \left[\mathbf{a}_i^\top \mathbf{u} + \rho_{w_i} b_i \right]^{v_{i}} \left[1-\mathbf{a}_i^\top \mathbf{u} - \rho_{w_i} b_i \right]^{1-v_{i}}
\end{align}
specifies the transition probability from state $\mathbf{u}$ to state $\mathbf{v}$.
We denote by  $\pi\in \mathbb{R}^{2^p}$ the associated stationary distribution with component $\pi_{\mathbf{u}}$ corresponding to the state $\mathbf{u} \in \{0,1\}^p$ and by $\mathbf{P}=\left(p_{\mathbf{u}\mathbf{v}}\right) \in \mathbb{R}^{2^p\times 2^p}$ the BAR transition probability matrix. 

The goal is to recover the model parameters from an observed sequence $\{\mathbf{X}(k)=\mathbf{x}(k)\}_{k=0}^{T}$.
Clearly, by inferring $\mathbf{A}$, estimates of $\mathbf{b}$ and  the underlying network structure  are direct per (\ref{eq: BAR_const}) and (\ref{eq: edge}), respectively.

\section{MAXIMUM LIKELIHOOD ESTIMATION}
\label{sec:ML Estimation}

In this section, we consider recovering the BAR model parameters via ML estimation and we establish the strong consistency of the ML estimator.
 Suppose that $\{\mathbf{x}(k)\}_{k=0}^{T}$ is a  sequence of observations generated by the BAR model (\ref{eq: BAR_def}). Let $\theta=(\mathbf{A},\mathbf{b}, \rho_{w})$ with the implicit relationship $\mathbf{b}=\mathbf{1}_p-\mathbf{A1}_p$. Clearly, $\mathbf{b}$ is a redundant parameter, but it is preserved here to facilitate the subsequent analysis. From (\ref{eq: tranProb}) the rescaled log-likelihood function is given by
\begin{equation}\label{eq: log_likelihood_1}
\begin{split}
    L_T(\theta) &=\frac{1}{T}\sum_{k=0}^{T-1} \log P\left(\mathbf{x}(k+1)|\mathbf{x}(k); \theta\right)+\frac{1}{T}\log P_{\mathbf{X}(0)}(\mathbf{x}(0);\mathbf{\theta})\\
    &=\frac{1}{T} \sum_{k=0}^{T-1} \sum_{i=1}^p \Bigg[ x_i(k+1)\log \left(\mathbf{a}_i^\top \mathbf{x}(k)+ \rho_{w_i} b_i\right) \\
    &+ \left( 1-x_i(k+1)\right)\log \left(1-\mathbf{a}_i^\top \mathbf{x}(k)- \rho_{w_i} b_i\right) \Bigg]\\
    & +\frac{1}{T}\log P_{\mathbf{X}(0)}(\mathbf{x}(0);\mathbf{\theta}).
\end{split}
\end{equation}
In the rest of the paper, we assume that $P_{\mathbf{X}(0)}$ is independent of the model parameters, which is well-aligned with the realistic scenario of arbitrarily initializing the Markov chain.

For any states $\mathbf{u}$, $\mathbf{v}\in \{0,1\}^p$, we denote by $N_{\mathbf{u}\mathbf{v}}$ the number of one-step transitions from state $\mathbf{u}$ to state $\mathbf{v}$ in the observed sequence and  we let $N_{\mathbf{u}}= \sum_{\mathbf{v}} N_{\mathbf{u}\mathbf{v}}$ be the amount of time spent in  state $\mathbf{u}$ in a horizon of $T$ time steps. Then (\ref{eq: log_likelihood_1}) can be also written as
\begin{equation}\label{eq:log_likelihood_1_new}
    L_T(\theta) = \sum_{\mathbf{u},\,\mathbf{v}}\frac{N_{\mathbf{u}\mathbf{v}}}{T}\log p_{\mathbf{u}\mathbf{v}}(\theta)+\frac{1}{T}\log P_{\mathbf{X}(0)}(\mathbf{x}(0)).
\end{equation}
Let $\theta_0$ be the true  parameter tuple $(\mathbf{A},\mathbf{b},\rho_w)$. An application of the Ergodic Theorem \cite{b13} for Markov chains  reveals that \[L_T(\theta_0)\xrightarrow[T\rightarrow \infty]{\rm a.s.} \sum_{\mathbf{u},\mathbf{v}}\pi_{\mathbf{u}}p_{\mathbf{uv}}(\theta_0)\log p_{\mathbf{uv}}(\theta_0),\] which is the negative of the entropy rate of the corresponding BAR chain with parameter set  $\theta_0$ and is always finite since the BAR model has a finite state space.   

Let $\theta_0\in \Theta$, where $\Theta$ is a compact set appropriately defined on the basis of (\ref{eq: BAR_const}) and additional assumptions in Section \ref{sec:BAR Model}. 
The ML estimator $\hat{\theta}_T$ of $\theta_0$ satisfies
\begin{equation}\label{eq:MLEstimatorDef}
     \hat{\theta}_{T} \in \arg\max_{\theta\in \Theta}\,\, T\cdot L_T(\theta) = \arg\max_{\theta\in \Theta}\,\, L_T(\theta).
\end{equation}
In the rest of this section, we will show the strong consistency of $\hat{\theta}_T$.
The key idea of the proof is along the lines of the proof of Theorem 2.1 in \cite{ranneby1978necessary} using techniques for general discrete-time Markov chains. To summarize, we first prove that $\mathbf{P}(\hat{\theta}_T)\xrightarrow[T\rightarrow \infty]{\rm a.s.}
\mathbf{P}(\theta_0)$. To establish that $\hat{\theta}_T\xrightarrow[T\rightarrow \infty]{\rm a.s.} \theta_0$, we then show that the vector-valued mapping $\mathbf{p}:\Theta\rightarrow \mathbb{R}^{2^{2p}}$ defined as $\mathbf{p}(\theta)=\mathrm{vec}(\mathbf{P}(\theta))$ is injective, i.e.,
\begin{equation}\label{eq: injectivity}
   \forall \theta, \theta' \in \Theta,\ \ \   \theta \neq \theta'\implies \mathbf{p}(\theta)\neq\mathbf{p}(\theta').
\end{equation} 
Here, $\mathrm{vec}(\cdot)$ denotes the vectorization of a matrix.
 Finally, we complete the proof by leveraging the compactness of the parameter space $\Theta$ and the continuity of the components of $\mathbf{p}(\theta)=\mathrm{vec}(\mathbf{P}(\theta))$ or equivalently, of the transition probabilities with respect to the model parameters. 

\textbf{Remark}: In the following, we will say that the BAR model is \emph{identifiable} when (\ref{eq: injectivity}) holds.

The main result of this section can be now stated.
\begin{theorem}\label{thm: consistency}
The ML estimator $\hat{\theta}_T$ of $\theta_0$, defined in (\ref{eq:MLEstimatorDef}), for the BAR model in (\ref{eq: BAR_def}) is strongly consistent.
\end{theorem}
\begin{proof}  We break up the proof into three parts.
\textbf{\emph{Proof of $\mathbf{P}(\hat{\theta}_T)\xrightarrow[T\rightarrow \infty]{\rm a.s.} \mathbf{P}(\theta_0)$}}: We present a simpler, self-contained proof, following ideas in proof of Theorem 2.1 in \cite{ranneby1978necessary}.

For each $\mathbf{u} \in \{0,1\}^p$, we define the (row) vector $\mathbf{Q}_{\mathbf{u}}=\left(N_{\mathbf{u}\mathbf{v}}/N_{\mathbf{u}}\right)_{\mathbf{v}\in\{0,1\}^p} \in\mathbb{R}^{2^p}$ with the convention $\mathbf{Q}_{\mathbf{u}}=2^{-p}\mathbf{1}_{2^p}^\top$ for $N_{\mathbf{u}}=0$ and we let $\mathbf{P}_{\mathbf{u}}= \left(P_{\mathbf{u}\mathbf{v}}\right)_{\mathbf{v}\in\{0,1\}^p} \in \mathbb{R}^{2^p}$ denote the transition distribution out of state $\mathbf{u}$, which is also a row vector in the transition matrix $\mathbf{P}$. In particular, it is well-known that the set $\{\mathbf{Q}_{\mathbf{u}}\}_{\mathbf{u}\in \{0,1\}^p}$ is the \emph{ML estimator} of the transition matrix $\mathbf{P}$, assuming no further parameterization of the transition probabilities. Consider the fact that the Kullback--Leibler divergence from $\mathbf{P}_{\mathbf{u}}(\hat{\theta}_T)$ to $\mathbf{Q}_{\mathbf{u}}$ is nonnegative, i.e.,
\begin{equation*}
    D_{\rm KL}\left(\mathbf{Q}_{\mathbf{u}}\Big\| \mathbf{P}_{\mathbf{u}}(\hat{\theta}_T)\right) = -\sum_{\mathbf{v}} \frac{N_{\mathbf{u}\mathbf{v}}}{N_{\mathbf{u}}} \log \frac{ p_{\mathbf{u}\mathbf{v}}(\hat{\theta}_T)}{N_{\mathbf{u}\mathbf{v}}/N_{\mathbf{u}}} \geq 0
\end{equation*}
or equivalently,
\begin{equation*}
    \sum_{\mathbf{v}} \frac{N_{\mathbf{u}\mathbf{v}}}{N_{\mathbf{u}}} \log \frac{N_{\mathbf{u}\mathbf{v}}}{N_{\mathbf{u}}} \geq \sum_{\mathbf{v}}  \frac{N_{\mathbf{u}\mathbf{v}}}{N_{\mathbf{u}}} \log p_{\mathbf{u}\mathbf{v}}(\hat{\theta}_T).
\end{equation*}
Multiply both sides of the above inequality by $\frac{N_{\mathbf{u}}}{T}$ and sum over $\mathbf{u}$ to obtain
\begin{equation}\label{eq: thm1_1}
    \sum_{\mathbf{u},\mathbf{v}} \frac{N_{\mathbf{u}\mathbf{v}}}{T} \log \frac{N_{\mathbf{u}\mathbf{v}}}{N_{\mathbf{u}}} \geq \sum_{\mathbf{u},\mathbf{v}} \frac{N_{\mathbf{u}\mathbf{v}}}{T} \log p_{\mathbf{u}\mathbf{v}}(\hat{\theta}_T) \geq \sum_{\mathbf{u},\mathbf{v}} \frac{N_{\mathbf{u}\mathbf{v}}}{T} \log p_{\mathbf{u}\mathbf{v}}(\theta_0)
\end{equation}
where the last inequality is due to (\ref{eq:log_likelihood_1_new}) and the definition of the ML estimator.
From (\ref{eq: thm1_1}), we can further obtain
\begin{equation}\label{eq: thm1_2}
    0\geq  \sum_{\mathbf{u},\mathbf{v}} \frac{N_{\mathbf{u}\mathbf{v}}}{T} \log \frac{p_{\mathbf{u}\mathbf{v}}(\hat{\theta}_T)}{N_{\mathbf{u}\mathbf{v}}/N_{\mathbf{u}}} \geq \sum_{\mathbf{u},\mathbf{v}} \frac{N_{\mathbf{u}\mathbf{v}}}{T} \log \frac{p_{\mathbf{u}\mathbf{v}}(\theta_0)}{N_{\mathbf{u}\mathbf{v}}/N_{\mathbf{u}}}.
\end{equation}
By the Ergodic Theorem for Markov chains, 
\begin{equation} \label{eq: ergodic_thm}
      \sum_{\mathbf{u},\,\mathbf{v}}\frac{N_{\mathbf{u}\mathbf{v}}}{T}\log p_{\mathbf{u}\mathbf{v}}(\theta_0) \xrightarrow[T\rightarrow \infty]{\rm a.s.} \sum_{\mathbf{u},\mathbf{v}} \pi_{\mathbf{u}}(\theta_0) p_{\mathbf{u}\mathbf{v}}(\theta_0) \log p_{\mathbf{u}\mathbf{v}}(\theta_0)
\end{equation}
and also
\begin{equation}\label{eq: thm1_3}
     \sum_{\mathbf{u},\mathbf{v}} \frac{N_{\mathbf{u}\mathbf{v}}}{T} \log \frac{N_{\mathbf{u}\mathbf{v}}}{N_{\mathbf{u}}} \xrightarrow[T\rightarrow \infty]{\rm a.s.} \sum_{\mathbf{u},\mathbf{v}} \pi_{\mathbf{u}}(\theta_0) p_{\mathbf{u}\mathbf{v}}(\theta_0) \log p_{\mathbf{u}\mathbf{v}}(\theta_0).
\end{equation}
By (\ref{eq: ergodic_thm}) and (\ref{eq: thm1_3}) we have that
\begin{equation*}
    \sum_{\mathbf{u},\mathbf{v}} \frac{N_{\mathbf{u}\mathbf{v}}}{T} \log \frac{p_{\mathbf{u}\mathbf{v}}(\theta_0)}{N_{\mathbf{u}\mathbf{v}}/N_{\mathbf{u}}}\xrightarrow[T\rightarrow \infty]{\rm a.s.}  0.
\end{equation*}
This together with (\ref{eq: thm1_2}) yields
\begin{equation}\label{eq: thm1_4}
     \sum_{\mathbf{u},\mathbf{v}} \frac{N_{\mathbf{u}\mathbf{v}}}{T} \log \frac{p_{\mathbf{u}\mathbf{v}}(\hat{\theta}_T)}{N_{\mathbf{u}\mathbf{v}}/N_{\mathbf{u}}}\xrightarrow[T\rightarrow \infty]{\rm a.s.} 0.
\end{equation}
Employing Pinsker's inequality \cite{kullback1967lower} and the fact that the total variation distance between two discrete measures $q,r$ with vector forms $\mathbf{q},\mathbf{r}$, respectively, is $\|q-r\|_{\rm TV}=(1/2)\|\mathbf{q}-\mathbf{r}\|_1$,  we have that for each $\mathbf{u}\in \{0,1\}^p$
{\small \begin{equation*}
    \frac{1}{2}D_{\rm KL}\left(\mathbf{Q}_{\mathbf{u}}\Big\| \mathbf{P}_{\mathbf{u}}(\hat{\theta}_T)\right) \geq \left\| \mathbf{Q}_{\mathbf{u}} -\mathbf{P}_{\mathbf{u}}(\hat{\theta}_T) \right\|_{\rm TV}^2 \geq \frac{1}{4}\left\| \mathbf{Q}_{\mathbf{u}} -\mathbf{P}_{\mathbf{u}}(\hat{\theta}_T) \right\|_2^2.
\end{equation*}}
Multiplying again with $\frac{N_{\mathbf{u}}}{T}$ and summing over $\mathbf{u}$ gives
\begin{equation}\label{eq: thm1_5}
    -2\sum_{\mathbf{u},\mathbf{v}} \frac{N_{\mathbf{u}\mathbf{v}}}{T} \log \frac{ p_{\mathbf{u}\mathbf{v}}(\hat{\theta}_T)}{N_{\mathbf{u}\mathbf{v}}/N_{\mathbf{u}}} \geq \sum_{\mathbf{u},\mathbf{v}}\frac{N_{\mathbf{u}}}{T} \left(p_{\mathbf{u}\mathbf{v}}(\hat{\theta}_T)-\frac{N_{\mathbf{u}\mathbf{v}}}{N_{\mathbf{u}}}\right)^2\geq0.
\end{equation}
Now employing again the Ergodic Theorem for Markov chains, i.e.,
\begin{equation*}
    \frac{N_{\mathbf{u}}}{T}\xrightarrow[T\rightarrow \infty]{\rm a.s.}\pi_{\mathbf{u}}>0, \forall \mathbf{u}\in \{0,1\}^p,
\end{equation*}
and combining (\ref{eq: thm1_4}) and (\ref{eq: thm1_5}) yields 
\begin{equation*}
    \left|p_{\mathbf{u}\mathbf{v}}(\hat{\theta}_T)-\frac{N_{\mathbf{u}\mathbf{v}}}{N_{\mathbf{u}}}\right| \xrightarrow[T\rightarrow \infty]{\rm a.s.} 0, \ \ \forall (\mathbf{u},\mathbf{v})\in (\{0,1\}^p)^2.
\end{equation*}
We then end up with
\begin{equation}\label{eq: thm1_6}
    \left|p_{\mathbf{u}\mathbf{v}}(\hat{\theta}_T)-p_{\mathbf{u}\mathbf{v}}(\theta_0)\right|  \xrightarrow[T\rightarrow \infty]{\rm a.s.} 0, \ \ \forall (\mathbf{u},\mathbf{v})\in (\{0,1\}^p)^2.
\end{equation}

\noindent\textbf{\emph{Proof of identifiability}}:  Due to the redundancy of $\mathbf{b}$, we reparameterize the BAR model as $\mathbf{\theta}=(\mathbf{A},\mathbf{c})$ with $\mathbf{c}=\mathrm{diag}(\mathbf{b})\mathbf{\rho}_w$. Clearly, there is a one-to-one correspondence between a given set of parameters $(\mathbf{A},\mathbf{c})$ and $\left(\mathbf{A},\mathbf{b}, \rho_w\right)$ via the relations $\mathbf{b}=\mathbf{1}_p-\mathbf{A1}_p$ and $\mathbf{\rho}_w=\left(\mathrm{diag}(\mathbf{1}_p-\mathbf{A1}_p)\right)^{-1}\mathbf{c}$, where $(\cdot)^{-1}$ denotes matrix inversion.
Suppose that two different sets of parameters $\theta = \left(\mathbf{A},\mathbf{c}\right)$ and $\theta' = \left(\mathbf{A}',\mathbf{c}'\right)$ lead to the same transition probability matrix, i.e., $\mathbf{P}(\theta)=\mathbf{P}(\theta')$ or equivalently, $\mathbf{p}(\theta)=\mathbf{p}(\theta')$.  
First, consider the case of $\mathbf{c}\neq\mathbf{c}'$. The following argument is valid for both the cases of $\mathbf{A}=\mathbf{A}'$ and $\mathbf{A}\neq\mathbf{A}'$. Without loss of generality, assume that $c_1 \neq c_1'$. Let $\mathbf{u}=\mathbf{0}$ and $\mathbf{v}$ be some vector in  $\{0,1\}^p$ with $v_1 = 0$. Since $p_{\mathbf{uv}}(\theta)=p_{\mathbf{uv}}(\theta')$ and $c_i,c_i'\neq 0 $ $\forall i$, (\ref{eq: tranProb}) implies that
\begin{equation}\label{eq: lemma1_1}
1-c_1' = \left(1-c_1\right) \prod_{i=2}^p \left(\frac{c_i}{c'_i} \right)^{v_i} \left(\frac{1-c_i}{1-c'_i} \right)^{1-v_i}.
\end{equation}

Consider now the transition probability from $\mathbf{u}=\mathbf{0}$ to $\mathbf{v}'$, where $v'_1=1$ and $v'_j = v_j$ for $j=2,\dots,p$. Since $p_{\mathbf{uv'}}(\theta)=p_{\mathbf{uv'}}(\theta')$,
\begin{equation}\label{eq: lemma1_2}
    c'_1 =  c_1 \prod_{i=2}^p \left(\frac{c_i}{c'_i} \right)^{v_i} \left(\frac{1- c_i}{1- c'_i} \right)^{1-v_i}.
\end{equation}
Combining (\ref{eq: lemma1_1}) and (\ref{eq: lemma1_2}), it is easy to see that $c_1 = c_1'$, which is a contradiction. Thus, $\mathbf{c}\neq \mathbf{c'}\implies \mathbf{P(\theta)}\neq \mathbf{P}(\theta')$ or equivalently, $\mathbf{p}(\theta)\neq \mathbf{p}(\theta')$.  

Now we consider the second case where $\mathbf{c}=\mathbf{c'}$ and $\mathbf{A}\neq \mathbf{A}'$. Without loss of generality, let $a_{11}\neq a'_{11}$. Consider $\mathbf{u}'=\mathbf{e}_{p,1}$ and the same $\mathbf{v},\mathbf{v}'$ as before. By our assumption that $p_{\mathbf{u'v}}(\theta)=p_{\mathbf{u'v}}(\theta')$ and $p_{\mathbf{u'v'}}(\theta)=p_{\mathbf{u'v'}}(\theta')$, the contradiction $a_{11}= a'_{11}$ arises. Thus, $\mathbf{c}=\mathbf{c'}, \mathbf{A}\neq \mathbf{A}'\implies \mathbf{P(\theta)}\neq \mathbf{P}(\theta')$ or equivalently, $\mathbf{p}(\theta)\neq \mathbf{p}(\theta')$.  

Finally, it is easy to see that $\mathbf{c}=\mathbf{c'}$ and $\mathbf{A} = \mathbf{A}'$ imply that $\mathbf{b}=\mathbf{b}'$ and $\mathbf{\rho}_w=\mathbf{\rho}_w'$  due to the aforementioned one-to-one correspondence between $(\mathbf{A},\mathbf{c})$ and $\left(\mathbf{A},\mathbf{b}, \rho_w\right)$.

\noindent\textbf{\emph{Completion of the proof}}: Let $\left(\Omega, \mathcal{F}, P_{\theta_0, P_{\mathbf{X}(0)}}\right)$ be the probability space on which the BAR process is defined, where the subscript $P_{\mathbf{X}(0)}$ indicates that the law of the BAR chain depends on the initial measure. By (\ref{eq: thm1_6}),
 \begin{equation*}
    P_{\theta_0, P_{\mathbf{X}(0)}}\left(\tilde{\Omega}=\left\{\omega\in \Omega: \lim_{T\rightarrow\infty} P(\hat{\theta}_T(\omega))= P(\theta_0)\right\}\right)=1,
\end{equation*}
which holds independently of $P_{\mathbf{X}(0)}$ due to the Ergodic Theorem for Markov chains.
Consider $\omega\in \tilde{\Omega}$ such that  $ \lim_{T\rightarrow\infty} \hat{\theta}_T(\omega)= \theta_0$ does not hold and note that $\hat{\theta}_T(\omega)\in \Theta, \forall T\geq 1$ by (\ref{eq:MLEstimatorDef}). Then, by the Bolzano-Weierstrass Theorem, there exists a subsequence $\{\hat{\theta}_{T_{k}}(\omega)\}_{k=1}^\infty$ converging to some point $\theta^{*}(\omega)\neq \theta_0$. Here, we use the fact that if  every convergent subsequence of the bounded sequence $\{\hat{\theta}_T(\omega)\}_{T=1}^{\infty}$ converges to the same limit $\theta_0$, then $ \lim_{T\rightarrow\infty} \hat{\theta}_T(\omega)= \theta_0$.
The compactness of $\Theta$ implies that $\theta^{*}(\omega)\in \Theta$. Since for every pair $(\mathbf{u},\mathbf{v})$,  $p_{\mathbf{u}\mathbf{v}}(\theta)$ is continuous in $\theta$ due to (\ref{eq: tranProb}) and the definition of $\Theta$, $p_{\mathbf{u}\mathbf{v}}(\theta)$ is sequentially continuous. Therefore,
$
    \lim_{k\rightarrow\infty} p_{\mathbf{u}\mathbf{v}}(\hat{\theta}_{T_k}(\omega))= p_{\mathbf{u}\mathbf{v}}(\theta^*(\omega)), \forall (\mathbf{u},\mathbf{v})$
and $p_{\tilde{\mathbf{u}}\tilde{\mathbf{v}}}(\theta^*(\omega))\neq p_{\tilde{\mathbf{u}}\tilde{\mathbf{v}}}(\theta_0)$ for at least a pair $(\tilde{\mathbf{u}},\tilde{\mathbf{v}})$. Here, the identifiability of the BAR model is invoked. Observing now that $\left\{p_{\tilde{\mathbf{u}}\tilde{\mathbf{v}}}(\hat{\theta}_{T_k}(\omega))\right\}_{k=1}^{\infty}$ is a subsequence of $\left\{p_{\tilde{\mathbf{u}}\tilde{\mathbf{v}}}(\hat{\theta}_{T}(\omega))\right\}_{T=1}^{\infty}$, a contradiction with  the choice of $\omega\in \tilde{\Omega}$ for which $ \lim_{T\rightarrow\infty} p_{\mathbf{u}\mathbf{v}}(\hat{\theta}_T(\omega))= p_{\mathbf{u}\mathbf{v}}(\theta_0), \forall (\mathbf{u}, \mathbf{v})$ is established. 
Therefore, $ \lim_{T\rightarrow\infty} \hat{\theta}_T(\omega)= \theta_0, \forall \omega \in \Tilde{\Omega}$, or equivalently  $\hat{\theta}_T\xrightarrow[T\rightarrow \infty]{\rm a.s}\theta_0$.
\end{proof}

\section{A CLOSED-FORM ESTIMATOR}
\label{sec:closed form}

In this section, we provide a closed-form estimator for the BAR model parameters in (\ref{eq: BAR_def}), i.e., for $(\mathbf{A},\mathbf{c})$, since by knowing $(\mathbf{A},\mathbf{c})$ we can recover $(\mathbf{b},\mathbf{\rho}_w)$. Recall that $\mathbf{c}=\mathrm{diag}(\mathbf{b})\mathbf{\rho}_w$.  Considering the log-likelihood function for $\theta=(\mathbf{A},\mathbf{c})$, we have
\begin{align} \label{eq: log-likelihood_3}
    L(\theta) = & \sum_{k=0}^{T-1} \sum_{\mathbf{u}}\sum_{\mathbf{v}} \mathbb{I}\left(\mathbf{x}(k)=\mathbf{u},\,\mathbf{x}(k+1)=\mathbf{v}\right) \log p_{\mathbf{u}\mathbf{v}}(\theta)\nonumber\\
    &+ \sum_{\mathbf{u}} \mathbb{I}\left(\mathbf{X}(0)=\mathbf{u}\right) \log P_{\mathbf{X}(0)}(\mathbf{u})\nonumber\\
     = &  \sum_{k=0}^{T-1} \sum_{\mathbf{u}}\sum_{\mathbf{v}} \mathbb{I}\left(\mathbf{x}(k)=\mathbf{u},\,\mathbf{x}(k+1)=\mathbf{v}\right) \cdot \nonumber\\
    & \sum_{r=1}^p \left[v_r \log P\left(v_r=1|\mathbf{u}\right) + \left(1-v_r\right)\log P\left(v_r=0|\mathbf{u}\right)\right]\nonumber\\
     &+ \sum_{\mathbf{u}} \mathbb{I}\left(\mathbf{X}(0)=\mathbf{u}\right) \log P_{\mathbf{X}(0)}(\mathbf{u}). 
\end{align}
Observe that $P\left(v_r=1|\mathbf{u}\right)$ and $P\left(v_r=0|\mathbf{u}\right)$ are independent of $\mathbf{v}$. We can therefore define $\vartheta_{\mathbf{u},r,l} = P\left((\cdot)_r = l|\mathbf{u}\right)$, for $l\in\{0,1\}$. Furthermore, we define $N_{\mathbf{u},r,l}=\sum_{k=0}^{T-1} \mathbb{I}\left(\mathbf{x}(k)=\mathbf{u}, x_r(k+1)=l\right)$, which is the number of times the BAR chain transitions from state $\mathbf{u}$ to a state with $r$-th entry being equal to $l$. Moreover, $N_{\mathbf{u},r,0}+ N_{\mathbf{u},r,1} = N_{\mathbf{u}} = \sum_{\mathbf{v}}N_{\mathbf{u}\mathbf{v}}$, $\forall \mathbf{u}\in\{0,1\}^p$ and $r\in[p]$. With these introductions we have the following theorem:

\begin{theorem}\label{thm: LS_red}
Consider an observed sequence $\{\mathbf{x}(k)\}_{k=0}^{T}$. For $i=1,\ldots, p$, define the estimator $\hat{\mathbf{c}}=[\hat{c}_1,\ldots, \hat{c}_p]^T$ by the entry estimators $\hat{c}_i=\sum_{k=0}^{T-1} \mathbb{I}(\mathbf{x}(k)=\mathbf{0}_p,x_i(k+1)=1)/\sum_{k=0}^{T-1} \mathbb{I}(\mathbf{x}(k)=\mathbf{0}_p)$, assuming that the state $\mathbf{0}_p$ is visited at least once in the time span $\{0,\ldots,T-1\}$. Moreover, in the special case where $\rho_{w_i}=\rho_w, \forall i\in \{1,\ldots, p\}$, $\hat{\mathbf{c}}$ can be replaced by $\hat{\mathbf{c}}=\left(\frac{1}{p}\sum_{i=1}^p\hat{c}_i\right)\mathbf{1}_p$.
Furthermore, suppose that in $\{\mathbf{x}(k)\}_{k=0}^{T-1}$ there are $m$ distinct states $\mathbf{u}_1$, $\mathbf{u}_2$,..., $\mathbf{u}_m$ such that $p\leq m\leq 2^p$. Let $\mathbf{U}_m\in \mathbb{R}^{m\times p}$ be a matrix with $k$-th row equal to $\mathbf{u}_k^\top$ for $k\in[m]$ and $\mathbf{y}_{m,r} = \left[N_{\mathbf{u}_1,r,1}/N_{\mathbf{u}_1},\cdots, N_{\mathbf{u}_m,r,1}/N_{\mathbf{u}_m}\right]^\top$. Then, whenever $\mathbf{U}_{m}$ is full-column rank, $\hat{\mathbf{A}}$ is an estimator of $\mathbf{A}$, where
\begin{equation}\label{eq: LS_red}
    \hat{\mathbf{a}}_r = \left(\mathbf{U}_m^\top\mathbf{U}_m\right)^{-1}\mathbf{U}_m^\top\left(\mathbf{y}_{m,r}-\hat{c}_r\cdot \mathbf{1}_m\right), \ \ \forall r\in [p].
\end{equation}
Finally, to obtain a valid estimate of the parameter set for any $T\geq 1$, we let 
\[
\hat{\mathbf{\theta}}=\left[\left(\hat{\mathbf{A}}, \hat{\mathbf{b}}=\mathbf{1}_p-\hat{\mathbf{A}}\mathbf{1}_p,\hat{\mathbf{\rho}}_w= \left(\mathrm{diag}(\mathbf{1}_p-\hat{\mathbf{A}}\mathbf{1}_p)\right)^{-1}\hat{\mathbf{c}}\right)\right]^{+},
\]
where $[\cdot]^{+}$ corresponds to a projection onto the parameter space $\Theta$. 
\end{theorem}
\begin{proof}
First, let us rewrite the log-likelihood function in (\ref{eq: log-likelihood_3}) as 
\begin{align*}
       \mathcal{L}=L(\theta)=& \sum_{\mathbf{u}}\sum_{r=1}^p \left(N_{\mathbf{u},r,0} \log \vartheta_{\mathbf{u},r,0} +  N_{\mathbf{u},r,1} \log\vartheta_{\mathbf{u},r,1}\right)\\& + \sum_{\mathbf{u}} \mathbb{I}\left(\mathbf{X}(0)=\mathbf{u}\right) \log P_{\mathbf{X}(0)}(\mathbf{u}).
\end{align*}
Instead of maximizing this function with respect to $\theta=(\mathbf{A},\mathbf{c})$, we maximize it with respect to the choice of the marginal conditional probabilities $\{\vartheta_{\mathbf{u},r,0},\vartheta_{\mathbf{u},r,1}\}_{\mathbf{u},r}$.
Consider the constrained ML estimation problem
\begin{align}\label{eq: const_MLE}
    & \quad\max_{\{\vartheta_{\mathbf{u},r,0}\geq 0, \vartheta_{\mathbf{u},r,1}\geq 0\}_{\mathbf{u},r}}  \quad \mathcal{L} \nonumber\\
   & {\rm s.t.} \quad \vartheta_{\mathbf{u},r,0} + \vartheta_{\mathbf{u},r,1} =1, \quad \forall \mathbf{u}\in\{0,1\}^p, \, \forall r\in[p].
\end{align}
Forming the Lagrangian and setting the gradient (with respect to $\{\vartheta_{\mathbf{u},r,0},\vartheta_{\mathbf{u},r,1}\}_{\mathbf{u},r}$) to zero, we obtain
\begin{equation*}
    \hat{\vartheta}_{\mathbf{u},r,i} = \frac{N_{\mathbf{u},r,i}}{N_{\mathbf{u}}},\quad \forall \mathbf{u}\in\{0,1\}^p,\, \forall r\in[p],\, \forall i\in\{0,1\}.
\end{equation*}
Recall that $\vartheta_{\mathbf{u},r,1}$ is defined as the probability of transitioning from state $\mathbf{u}$ to some state with  $r$-th component equal to $1$. We can therefore require that 
\begin{align}\label{eq:invariance}
    \begin{bmatrix}\frac{N_{\mathbf{u}_1,r,1}}{N_{\mathbf{u}_1}} \\  \vdots \\\frac{N_{\mathbf{u}_m,r,1}}{N_{\mathbf{u}_m}} \end{bmatrix} = \underbrace{\begin{bmatrix}\mathbf{u}_1^\top \\  \vdots \\\mathbf{u}_m^\top\end{bmatrix}}_{\mathbf{U}_m}\cdot \hat{\mathbf{a}}_r +\hat{c}_r\mathbf{1}_m
\end{align}
\begin{equation}\label{eq:LS_red_2}
\text{or}\ \ \mathbf{y}_{m,r}-\hat{c}_r\mathbf{1}_m=\mathbf{U}_m\hat{\mathbf{a}}_r,
\end{equation}
where the estimates $\hat{c}_r$ are provided in the statement of the theorem.
Under the assumption that $\mathbf{U}_m$ is full-column rank, $\mathbf{U}_m^\top \mathbf{U}_m$ is nonsingular and (\ref{eq: LS_red}) follows. The proof is then concluded by projecting the obtained estimate onto $\Theta$.
\end{proof}

\textbf{Remark}: (\ref{eq:invariance}) is reminiscent of the \emph{invariance property} for ML estimation.

\begin{theorem}\label{thm: LS_constistency}
The closed-form estimator in Theorem \ref{thm: LS_red} is strongly consistent.
\end{theorem}
\begin{proof}
It is sufficient to show that $(\hat{\mathbf{A}},\hat{\mathbf{c}})$ given by Theorem \ref{thm: LS_red} is strongly consistent.
Since the BAR chain is finite-state, the stationary probabilities satisfy $\pi_{\mathbf{u}}>0, \forall \mathbf{u}\in \{0,1\}^p$. Moreover, for any initial measure, the Ergodic Theorem for Markov chains implies that\footnote{The following convergences can be easily justified by straightforward embeddings of $\mathbf{U}_m$ into $\mathbb{R}^{2^p\times p}$ in the first case and of $\mathbf{y}_{m,r}$ into $\mathbb{R}^{2^p}$ in the second case via zero padding.}
\begin{itemize}
\item $\frac{N_\mathbf{u}}{T}\xrightarrow[T\rightarrow \infty]{\rm a.s.}\pi_{\mathbf{u}}$. This further implies that $\mathbf{U}_m\xrightarrow[T\rightarrow \infty]{\rm a.s.} \mathbf{U}_{2^p}$ in the sense that $[\mathbf{U}_m^\top\ \ \mathbf{0}_{2^p-m\times p}^\top]^\top\xrightarrow[T\rightarrow \infty]{\rm a.s.} \mathbf{U}_{2^p}$. 
\item $\frac{N_\mathbf{u},r,1}{N_{\mathbf{u}}}\xrightarrow[T\rightarrow \infty]{\rm a.s.} P\left((\cdot)_r = 1|\mathbf{u}\right), \forall \mathbf{u}\in \{0,1\}^p, \forall r\in [p]$ or equivalently, $\frac{N_\mathbf{u},r,1}{N_{\mathbf{u}}}\xrightarrow[T\rightarrow \infty]{\rm a.s.} \mathbf{u}^\top\mathbf{a}_r+c_r, \forall \mathbf{u}\in \{0,1\}^p, \forall r\in [p]$. This implies that $\mathbf{y}_{m,r}\xrightarrow[T\rightarrow \infty]{\rm a.s.} [P\left((\cdot)_r = 1|\mathbf{u}\right)]_{\mathbf{u}\in \{0,1\}^p}, \forall r\in [p]$, which is a $2^p\times 1$ column vector, in the sense that  $[\mathbf{y}_{m,r}^\top\ \ \mathbf{0}_{2^p-m}^\top]^\top\xrightarrow[T\rightarrow \infty]{\rm a.s.} [P\left((\cdot)_r = 1|\mathbf{u}\right)]_{\mathbf{u}\in \{0,1\}^p}, \forall r\in [p]$. As a consequence, $\hat{\mathbf{c}}\xrightarrow[T\rightarrow \infty]{\rm a.s.}\mathbf{c}$.
\end{itemize}

Combining these observations with (\ref{eq:LS_red_2}) we obtain
\begin{align}
&\lim_{T\rightarrow \infty} [\mathbf{U}_m^\top\ \ \mathbf{0}_{2^p-m\times p}^\top]^\top\hat{\mathbf{a}}_r=\lim_{T\rightarrow \infty} \mathbf{U}_{2^p}\hat{\mathbf{a}}_r=\mathbf{U}_{2^p}\mathbf{a}_r\ \ \text{a.s.}
\end{align}
and the strong consistency of the closed-form estimator in  Theorem \ref{thm: LS_red} follows if $\mathbf{U}_{2^p}$ is full-column rank or equivalently if $\mathbf{U}_{2^p}^\top\mathbf{U}_{2^p}$ is nonsingular.  It is easy to see that 
$\mathbf{U}_{2^p}^\top\mathbf{U}_{2^p}$ has diagonal entries equal to $2^{p-1}$ and off-diagonal entries equal to $2^{p-2}$. Thus, we can write $\mathbf{U}_{2^p}^\top\mathbf{U}_{2^p}=2^{p-2}\mathbf{1}_p\mathbf{1}_p^\top+2^{p-2}\mathbf{I}_p$. This matrix is invertible for every $p<\infty$, since $1+2^{p-2}\mathbf{1}_p^\top\left(2^{p-2}\mathbf{I}_p\right)^{-1}\mathbf{1}_p=1+p>0$ as the condition in the Sherman–Morrison formula \cite{press2007numerical} dictates.
\end{proof}

\section{THE GENERIC BAR MODEL}
\label{sec: General BAR}

Motivated by modeling positive and negative influences from parental nodes, an extension of the BAR model in (\ref{eq: BAR_def}) has been introduced in \cite{katselis2018mixing}. We first reformulate this generic BAR model.

Denote by $\mathcal{S}_i=\mathcal{S}_i^+ \cup \mathcal{S}_i^- $ the parental set of node $i$, where $\mathcal{S}_i^+ \cap \mathcal{S}_i^- = \emptyset$. The nodes in $\mathcal{S}_i^+$ and $\mathcal{S}_i^-$ are said to have positive and negative influence on $i$, respectively. The generic BAR model, parameterized by  $\tilde{\theta} = \left(\mathbf{A},\Tilde{\mathbf{A}},\mathbf{b},\rho_w\right)$, is defined as
\begin{equation}\label{eq: BAR_def2}
    X_i(k+1)\sim \textnormal{Ber}\left(\mathbf{a}_i^\top\mathbf{X}(k)+\Tilde{\mathbf{a}}_i^\top\left(\mathbf{1}-\mathbf{X}(k)\right)+b_iW_i(k+1)\right) ,
\end{equation}
for all $i\in[p]$, where $\mathbf{a}_i^\top$ and $\Tilde{\mathbf{a}}_i^\top$ are the $i$-th rows of $\mathbf{A}\in\mathbb{R}^{p\times p}$ and $\Tilde{\mathbf{A}}\in\mathbb{R}^{p\times p}$, respectively. Furthermore, we assume that $\mathcal{S}_i^+ = \text{supp}(\mathbf{a}_i)$ and $\mathcal{S}_i^- = \text{supp}(\Tilde{\mathbf{a}}_i)$. Here, $\text{supp}(\cdot)$ denotes the support  of a vector. As in the previous case, the constraints
\begin{equation}\label{eq: BAR_const2}
    \sum_{j=1}^p \left(a_{ij} + \Tilde{a}_{ij} \right) + b_i = 1, \quad \forall i\in[p]
\end{equation}
are also required  in this case. Similarly, we assume that $a_{ij},\Tilde{a}_{ij}\geq0$, $\forall i,j\in[p]$, $b\geq b_{min}$ and $\rho_{w_i}\in[\rho_{min},\,\rho_{max}], \forall i\in [p]$. Therefore, the parameter space is defined as 
\begin{equation*}
\begin{split}
    \tilde{\Theta} = \Bigg\{ & \left(\mathbf{A},\Tilde{\mathbf{A}},\mathbf{b}, \rho_w \right)\Big| \sum_{j=1}^p \left(a_{ij} + \Tilde{a}_{ij} \right)+ b_i = 1, \,\, b_i\geq  b_{min},\\ & \rho_{w_i}\in\left[\rho_{min},\,\rho_{max}\right],  \forall i\in [p], \, \, \text{and} \,\, \,a_{ij},\tilde{a}_{ij}\geq 0,\\ &   a_{ij}\Tilde{a}_{ij}=0,\,\forall i,j \in [p]  \Bigg\} .
\end{split}
\end{equation*}
\textbf{Remark}: The parameter space $\tilde{\Theta}$ is compact. To see this, first note that the associated constraints imposed on $\tilde{\Theta}$ are defined row-wise. This implies that $\tilde{\Theta}$ can be viewed as the Cartesian product of row spaces, i.e., $\tilde{\Theta}=\prod_{i=1}^p  \tilde{\Theta}_i$, where 
\begin{equation*}
\begin{split}
     \tilde{\Theta}_i =& \Bigg\{
     \left(\mathbf{a}_i^\top, \tilde{\mathbf{a}}_i^\top, b_i, \rho_{w_i}\right)\Big| \sum_{j=1}^p (a_{ij} + \Tilde{a}_{ij}) + b_i = 1,b_i\geq  b_{min},\\ & \rho_{w_i}\in\left[\rho_{min},\,\rho_{max}\right], \text{and} \,\, a_{ij},\tilde{a}_{ij}\geq 0,\, a_{ij}\Tilde{a}_{ij}=0,\,\forall j \in [p]  \Bigg\} .
     \end{split}
\end{equation*}
Then $\tilde{\Theta}$ is compact if and only if $\tilde{\Theta}_i$ is compact for every $i\in[p]$. It is straightforward to see that $\tilde{\Theta}_i$ is compact without the constraints $a_{ij}\Tilde{a}_{ij}=0$, $\forall j\in[p]$. Also observe that adding a constraint $a_{ij}\Tilde{a}_{ij}=0$ leads to the coordinate projection of a compact set in $\mathbb{R}^{2p+2}$ onto two $(2p+1)$-dimensional subspaces of $\mathbb{R}^{2p+2}$ with the corresponding images of the projected compact set being also compact sets. Denote the union of these images as $\tilde{\Theta}_{i,j}$. Then $\tilde{\Theta}_i = \cap_{j=1}^p\tilde{\Theta}_{i,j}$, i.e., $\tilde{\Theta}_i$ is the intersection of $p$ compact sets and is therefore compact. 

The ML estimator is a maximizer of the rescaled log-likelihood function, that is,
\begin{equation*}
    \tilde{\theta}_T \in \arg\max_{\tilde{\theta}\in \tilde{\Theta}}\,\,\Tilde{L}_T(\tilde{\theta}),
\end{equation*}
where
{\small 
\begin{equation*}
\begin{split}
     &\Tilde{L}_T(\tilde{\theta}) = \frac{1}{T} \sum_{k=0}^{T-1} \sum_{i=1}^p \Bigg[ x_i(k+1)\log \left(\mathbf{a}_i^\top \mathbf{x}(k)+ \Tilde{\mathbf{a}}_i^\top (\mathbf{1}-\mathbf{x}(k))+\rho_{w_i} b_i\right) \\
    &+ \left( 1-x_i(k+1)\right)\log \left(1-\mathbf{a}_i^\top \mathbf{x}(k)-\Tilde{\mathbf{a}}_i^\top (\mathbf{1}-\mathbf{x}(k))- \rho_{w_i} b_i\right) \Bigg] \\
    &+\frac{1}{T}\log P_{\mathbf{X}(0)}(\mathbf{x}(0)).
\end{split}
\end{equation*}
}
The ML estimator for the generic BAR model can be shown to be strongly consistent via a direct extension of the analysis in Section \ref{sec:ML Estimation}. More precisely, it is sufficient to establish identifiability.

\begin{theorem}\label{thm: identifiability}
For the generic BAR model in (\ref{eq: BAR_def2}), $\tilde{\theta} \neq \tilde{\theta}'\implies \mathbf{p}(\tilde{\theta}) \neq \mathbf{p}(\tilde{\theta}'), \forall (\tilde{\theta}, \tilde{\theta}' )\in \tilde{\Theta}\times \tilde{\Theta}$ with $\tilde{\theta} \neq \tilde{\theta}'$.
\end{theorem}
\begin{proof}
For two different sets of parameters $(\mathbf{A},\tilde{\mathbf{A}},\mathbf{b}, \rho_w)$, $(\mathbf{A}',\tilde{\mathbf{A}}',\mathbf{b}',\rho_w')$ and by letting $c_i = b_i \rho_{w_i}, i=1,\ldots, p$, we can reparameterize the generic BAR model to obtain $\tilde{\theta}=(\mathbf{A},\tilde{\mathbf{A}},\mathbf{c})$, $\tilde{\theta}'=(\mathbf{A}',\tilde{\mathbf{A}}',\mathbf{c}')$ by recalling the one-to-one correspondence between the initial sets of parameters and the later ones as in the case of the BAR model with only positive correlations. The corresponding relations in this case are $\mathbf{b}=\mathbf{1}_p-(\mathbf{A}+\tilde{\mathbf{A}})\mathbf{1}_p$ and $\mathbf{\rho}_w=\left(\mathrm{diag}(\mathbf{b})\right)^{-1}\mathbf{c}$. We will examine different cases for which $\tilde{\theta}\neq \tilde{\theta}'$ and by assuming that $\mathbf{p}(\tilde{\theta})=\mathbf{p}(\tilde{\theta}')$ we will arrive to contradictions. 
\begin{itemize}
    \item \emph{First Case}: Suppose that there exists some $i\in[p]$ such that $\mathcal{S}_i^{+} \neq \mathcal{S'}_i^{+}$, where $\mathcal{S}_i^{+}$ and $\mathcal{S’}_i^{+}$ correspond to the parental neighborhoods with positive influence on the $i$-th node, as these neighborhoods are encoded in $\tilde{\theta}$ and $\tilde{\theta}'$, respectively. Without loss of generality, we assume that $\mathcal{S}_1^{+} \neq \mathcal{S'}_1^{+}$. Translating this structural difference into the model parameters, suppose again without loss of generality that $1\in \mathcal{S}_1^+$, i.e., $a_{11}\neq 0$ and therefore,  $\Tilde{a}_{11}=0$ and $1\notin  \mathcal{S'}_1^+ $ , i.e., $a'_{11}=0$ and $\tilde{a'}_{11}=0$ or $\tilde{a'}_{11}>0$. 
    
    We first consider the transition probabilities from $\mathbf{u}=\sum_{i\in\mathcal{S}_1^-} \mathbf{e}_{p,i}$ to states $\mathbf{v}$ and $\mathbf{v}'$, where   $\mathbf{v}$ and $\mathbf{v}'$ only differ in the first element with $v_1=1$ and $v'_1=0$. Then by letting
    $p_{\mathbf{u}\mathbf{v}}(\tilde{\theta}) = p_{\mathbf{u}\mathbf{v}}(\tilde{\theta}')$ and $p_{\mathbf{u}\mathbf{v'}}(\tilde{\theta}) = p_{\mathbf{u}\mathbf{v'}}(\tilde{\theta}')$, we obtain
    \begin{align}\label{eq: thm2_1}
    c_1 &= \sum_{j\in\mathcal{S}_1^-}a'_{1j} + \sum_{k\notin \mathcal{S}_1^-}\Tilde{a'}_{1k} + c'_1 \nonumber\\
        &=\sum_{j\in\mathcal{S}_1^-}a'_{1j} + \tilde{a}'_{11} +\sum_{k\notin \mathcal{S}_1^-\cup\{1\}}\Tilde{a'}_{1k} + c'_1.
    \end{align}
    Further by considering the probabilities of transitioning from $\mathbf{u'}= \mathbf{e}_{p,1} + \sum_{i\in\mathcal{S}_1^-} \mathbf{e}_{p,i}$ to $\mathbf{v}$ and $\mathbf{v'}$, we obtain
    \begin{equation}\label{eq: thm2_2}
    \begin{split}
      &a_{11}+c_1=\sum_{j\in\mathcal{S}_1^-}a'_{1j} + \sum_{k\notin \mathcal{S}_1^-\cup \{1\}}\Tilde{a'}_{1k} + c'_1,
    \end{split}
    \end{equation}
    where the assumption $a'_{11}=0$ has been used.
    By (\ref{eq: thm2_1}) and (\ref{eq: thm2_2}) we have that $a_{11}+\tilde{a}'_{11}=0$, which implies that $a_{11}=0$ and contradicts our assumption. The case of  $\mathcal{S}_i^{-} \neq \mathcal{S'}_i^{-}$ is similar. 
    \item \emph{Second Case}: Suppose that $(\mathcal{S}_i^{+}, \mathcal{S}_i^{-}) =(\mathcal{S'}_i^{+},\mathcal{S’}_i^{-})$ for all $i\in[p]$ with either $\left(\mathbf{A},\mathbf{\Tilde{A}}\right)=\left(\mathbf{A}',\mathbf{\Tilde{A}}'\right)$ or $\left(\mathbf{A},\mathbf{\Tilde{A}}\right) \neq \left(\mathbf{A}',\mathbf{\Tilde{A}}'\right)$ and $\mathbf{c}\neq \mathbf{c'}$. Without loss of generality, let $c_1\neq c_1'$. Similarly to the proof of Theorem \ref{thm: consistency}, by selecting $\mathbf{u}=\sum_{i\in\mathcal{S}_1^-} \mathbf{e}_{p,i}$ and $\mathbf{v},\mathbf{v}'$ as before, we arrive at the contradiction  $c_1=c'_1$.
    \item \emph{Third Case}: Suppose that $(\mathcal{S}_i^{+}, \mathcal{S}_i^{-}) =(\mathcal{S'}_i^{+},\mathcal{S'}_i^{-})$ for all $i\in[p]$, $\mathbf{c} = \mathbf{c'}$ and $\left(\mathbf{A},\mathbf{\Tilde{A}}\right) \neq \left(\mathbf{A}',\mathbf{\Tilde{A}}'\right)$. Without loss of generality, we assume that $a_{11}\neq 0$, $a'_{11}\neq 0$ and $a_{11}\neq a'_{11}$. In this case, the contradiction $a_{11}= a'_{11}$ arises when selecting $\mathbf{u}= \mathbf{e}_{p,1} + \sum_{i\in\mathcal{S}_1^-} \mathbf{e}_{p,i}$ and $\mathbf{v},\mathbf{v}'$ as before.
\end{itemize}
\end{proof}

\section{A CLOSED-FORM ESTIMATOR FOR THE GENERIC BAR MODEL}
\label{sec: closed form 2}
Extending the closed-form estimator for the BAR model with only positive correlations, we now introduce a closed-form estimator for the generic BAR model. Consider the same introductions as before and also the Bernoulli argument for $X_i(k+1)$ in (\ref{eq: BAR_def2}). We can rewrite
\begin{align}\label{eq:BARargument-new}
&\mathbf{a}_i^\top\mathbf{X}(k)+\Tilde{\mathbf{a}}_i^\top\left(\mathbf{1}-\mathbf{X}(k)\right)+b_iW_i(k+1)=\nonumber\\
&(\mathbf{a}_i-\Tilde{\mathbf{a}}_i)^\top\mathbf{X}(k)+\Tilde{\mathbf{a}}_i^\top\mathbf{1}+b_iW_i(k+1)
\end{align}
and we note that due to the nonoverlapping  supports of $\mathbf{a}_i$ and $\Tilde{\mathbf{a}}_i$ for every $i\in [p]$, the vector $\mathbf{a}_i-\Tilde{\mathbf{a}}_i$ contains the entries of $\mathbf{a}_i$ and the entries of $\Tilde{\mathbf{a}}_i$ with flipped signs, each  at a different location. We further note that 
\begin{align}
P(X_i(k+1)=1|\mathbf{X}(k)=\mathbf{x}(k))&=(\mathbf{a}_i-\Tilde{\mathbf{a}}_i)^\top\mathbf{x}(k)+\Tilde{\mathbf{a}}_i^\top\mathbf{1}+b_i\rho_{w_i}\nonumber\\&=\bar{\mathbf{a}}_i^\top\mathbf{x}(k)+\bar{c}_i,
\end{align}
where $\bar{\mathbf{a}}_i=\mathbf{a}_i-\Tilde{\mathbf{a}}_i$ and $\bar{c}_i=\Tilde{\mathbf{a}}_i^\top\mathbf{1}+b_i\rho_{w_i}$ for $i\in[p]$. With these introductions, we can reparameterize the generic BAR model as $\bar{\mathbf{\theta}}=(\bar{\mathbf{A}},\bar{\mathbf{c}})$, where 
$\bar{\mathbf{A}}=[\bar{\mathbf{a}}_1,\ldots, \bar{\mathbf{a}}_p]^T$ and $\bar{\mathbf{c}}=[\bar{c}_1,\ldots, \bar{c}_p]^T$. Clearly, by knowing $\bar{\mathbf{A}}$, we can immediately separate $\mathbf{A}$ and $\Tilde{\mathbf{A}}$ based on the underlying signs of the entries. Furthermore, by also knowing $\bar{\mathbf{c}}$, we can compute the products $c_i=b_i\rho_{w_i}, i=1,\ldots, p$ and finally, we can compute $\rho_{w_i}$ via the formula
$\rho_{w_i}=c_i/(1-\sum_{j=1}^p(a_{ij}+\tilde{a}_{ij}))$ for $i\in [p]$. The parameterization $\bar{\mathbf{\theta}}=(\bar{\mathbf{A}},\bar{\mathbf{c}})$ is of the same form as the parameterization $\mathbf{\theta}=(\mathbf{A},\mathbf{c})$ of the BAR model with only positive correlations and therefore, the same unprojected closed-form estimator as before is suitable.

\begin{theorem}\label{thm: LS_red_genericBAR}
Consider an observed sequence $\{\mathbf{x}(k)\}_{k=0}^{T}$. For $i=1,\ldots, p$, define the estimator $\hat{\bar{\mathbf{c}}}=[\hat{\bar{c}}_1,\ldots, \hat{\bar{c}}_p]^T$ by the entry estimators $\hat{\bar{c}}_i=\sum_{k=0}^{T-1} \mathbb{I}(\mathbf{x}(k)=\mathbf{0}_p,x_i(k+1)=1)/\sum_{k=0}^{T-1} \mathbb{I}(\mathbf{x}(k)=\mathbf{0}_p)$, assuming that the state $\mathbf{0}_p$ is visited at least once in the time span $\{0,\ldots,T-1\}$. 
Furthermore, suppose that in $\{\mathbf{x}(k)\}_{k=0}^{T-1}$ there are $m$ distinct states $\mathbf{u}_1$, $\mathbf{u}_2$,..., $\mathbf{u}_m$ such that $p\leq m\leq 2^p$. Let $\mathbf{U}_m\in \mathbb{R}^{m\times p}$ be a matrix with $k$-th row equal to $\mathbf{u}_k^\top$ for $k\in[m]$ and $\mathbf{y}_{m,r} = \left[N_{\mathbf{u}_1,r,1}/N_{\mathbf{u}_1},\cdots, N_{\mathbf{u}_m,r,1}/N_{\mathbf{u}_m}\right]^\top$. Then, whenever $\mathbf{U}_{m}$ is full-column rank, $\hat{\bar{\mathbf{A}}}$ is an estimator of $\bar{\mathbf{A}}$, where
\begin{equation}\label{eq: LS_red2}
    \hat{\bar{\mathbf{a}}}_r = \left(\mathbf{U}_m^\top\mathbf{U}_m\right)^{-1}\mathbf{U}_m^\top\left(\mathbf{y}_{m,r}-\hat{\bar{c}}_r\cdot \mathbf{1}_m\right), \ \ \forall r\in [p].
\end{equation}
Based on the signs of the entries in  $\hat{\bar{\mathbf{A}}}$, we can separate $\hat{\mathbf{A}}$ and $\hat{\Tilde{\mathbf{A}}}$. Moreover, $\hat{\mathbf{b}}=\mathbf{1}_p-(\hat{\mathbf{A}}+\hat{\Tilde{\mathbf{A}}})\mathbf{1}_p$, $\hat{\mathbf{c}}=\hat{\bar{\mathbf{c}}}-\hat{\Tilde{\mathbf{A}}}\mathbf{1}_p$ and $\hat{\mathbf{\rho}}_w= \left(\mathrm{diag}(\hat{\mathbf{b}})\right)^{-1}\hat{\mathbf{c}}$. 
Finally, to obtain a valid estimate of the parameter set for any $T\geq 1$, we let 
$
\hat{\bar{\mathbf{\theta}}}=\left[\left(\hat{\mathbf{A}},\hat{\Tilde{\mathbf{A}}},  \hat{\mathbf{b}},\hat{\mathbf{\rho}}_w\right)\right]^{+},
$
where $[\cdot]^{+}$ corresponds to a projection onto the parameter space $\tilde{\Theta}$ by preserving the supports of $\hat{\mathbf{A}}$ and $\hat{\Tilde{\mathbf{A}}}$.  
\end{theorem}

The proofs of this theorem and of the strong consistency of the proposed closed-form estimator are straightforward based on the proofs of Theorems \ref{thm: LS_red}
and \ref{thm: LS_constistency}, respectively.

\section{SIMULATION RESULTS}
\label{sec:sims}

To validate our analysis, experiments on synthetic networks of various sizes are performed. Focusing only on the structure identification of the underlying BAR graphs, we compare the ML estimators with the closed-form estimators and the BAR structure observer (BARobs) proposed in \cite{katselis2018mixing}.  Our simulations show that the ML estimator either for the BAR model with positive correlations only or for the generic BAR model can perfectly recover the graph structure with sufficient data and outperforms the other estimators in terms of sample complexity.

\begin{figure}
\centering
\includegraphics[scale=0.255]{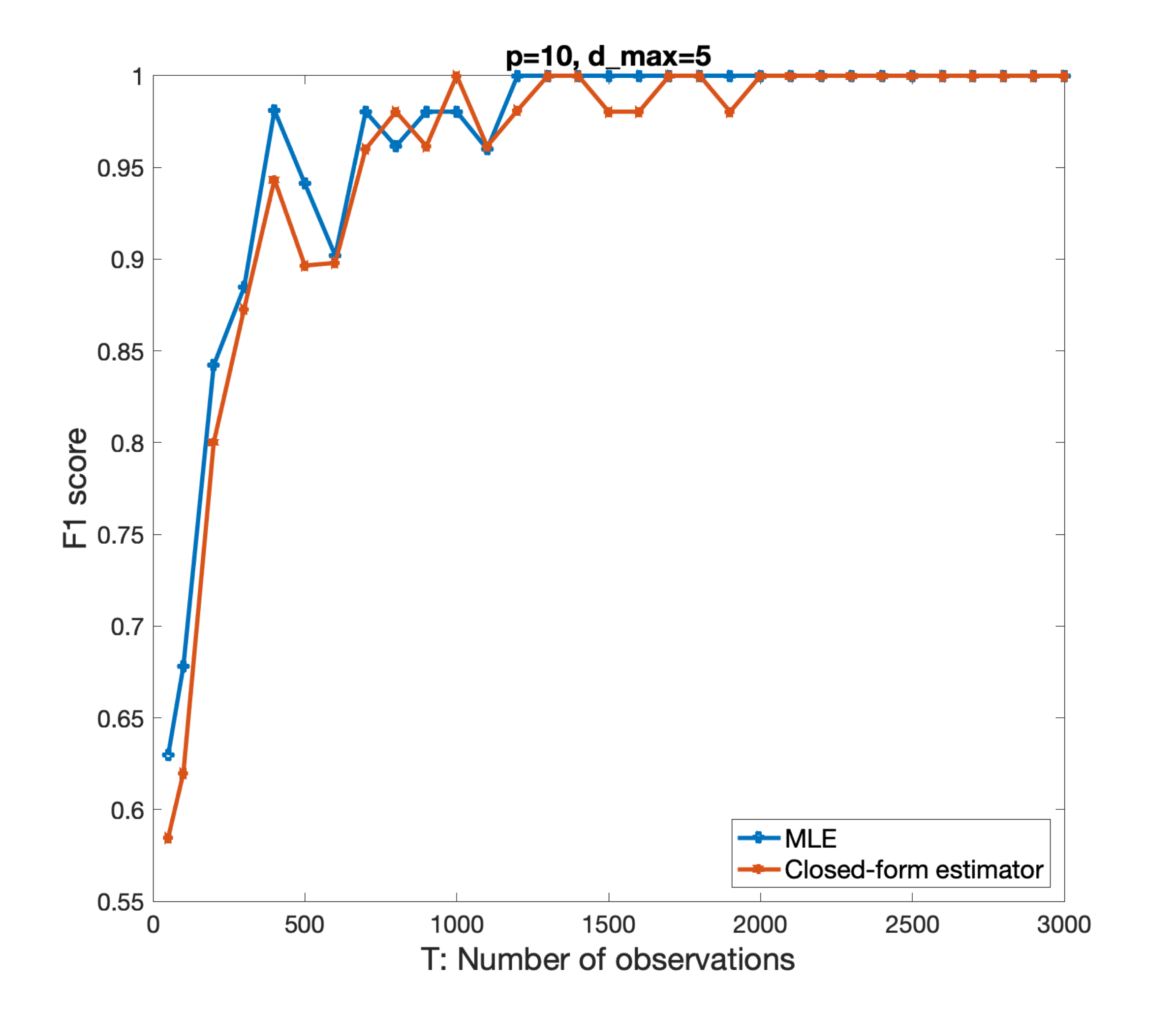}
\caption{$p=10$, $d_{max}=5$: $F_1$ scores for the ML and closed-form estimators.}
 \label{fig: MLEvsBARobs_a}
\end{figure} 

\begin{figure}
\centering
\includegraphics[scale=0.24]{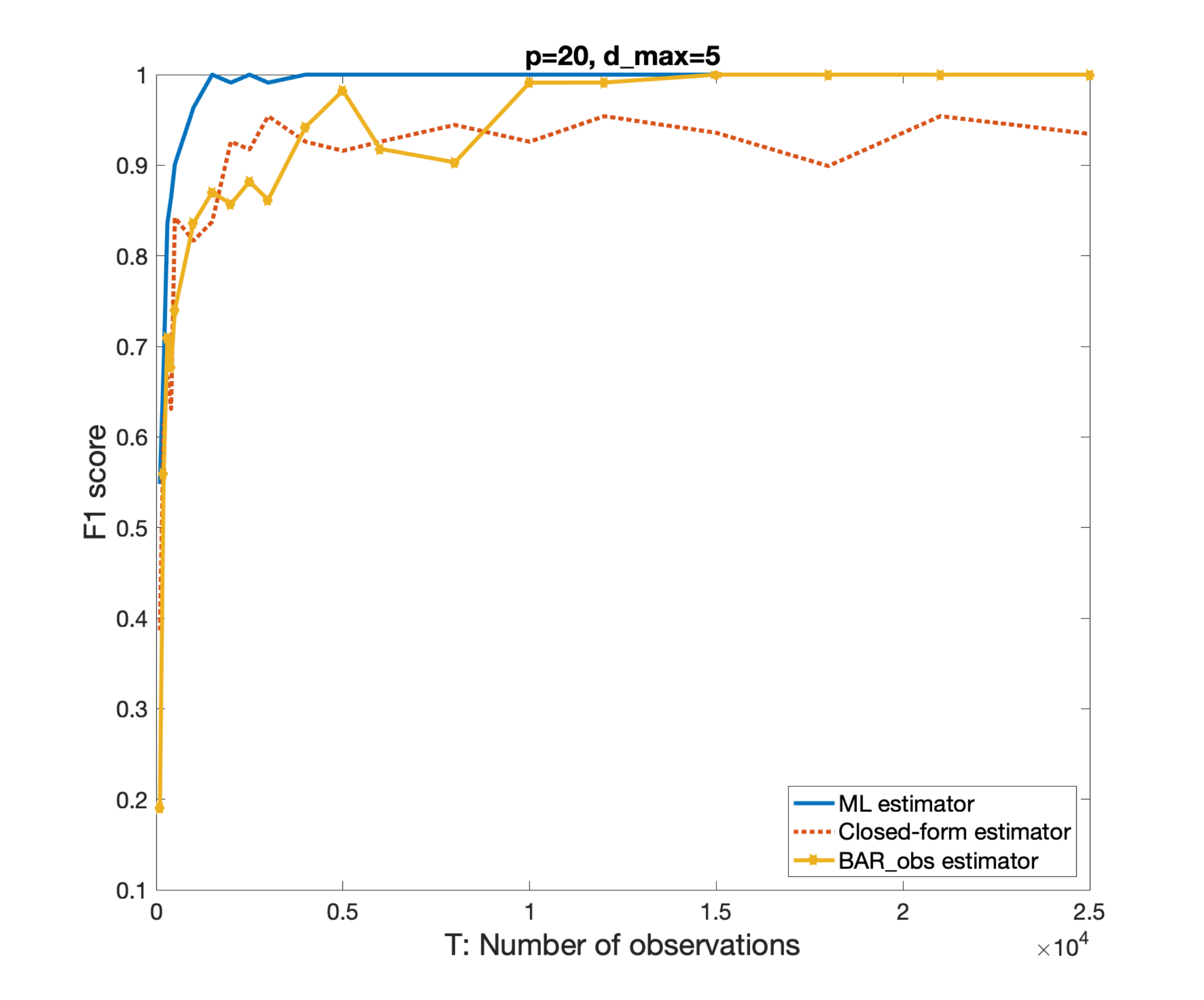}
\caption{$p=20$, $d_{max}=5$: $F_1$ scores for the ML,  closed-form and BARobs estimators.}
 \label{fig: MLEvsBARobs_c}
\end{figure} 

\begin{figure}
\centering
\includegraphics[scale=0.24]{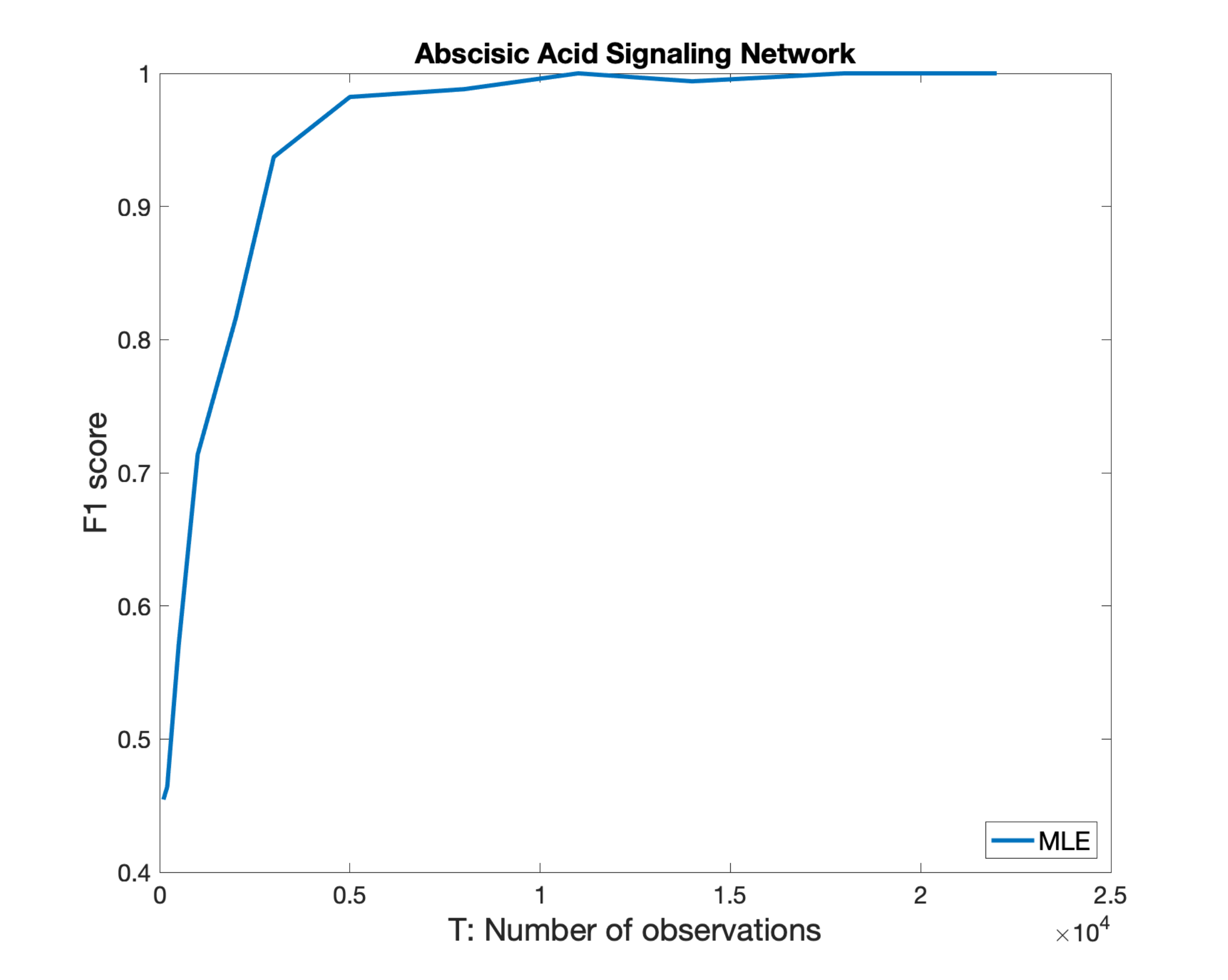}
\caption{A biological network with $43$ nodes: $F_1$ score for the ML estimator.}
   \label{fig: Bio}
\end{figure} 

We implement all experiments in MATLAB. In particular for the ML estimators, we use the ``fmincon'' function in the ``optimization toolbox'' to solve the maximization problems. The ground truth networks $\left(\mathbf{A},\mathbf{b},\rho_w\right)$ or $\left(\mathbf{A},\mathbf{\Tilde{A}},\mathbf{b},\rho_w\right)$ are randomly generated in such a way that all constraints specifying the parameter spaces $\Theta$ or $\Tilde{\Theta}$ are satisfied. Moreover, to facilitate the simulation, we create networks with minimum edge weight $a_{min}\in (0,1)$, i.e., $a_{ij}\geq a_{min}, \forall (j,i)\in\mathcal{E}$ in the case of the BAR model (\ref{eq: BAR_def}) and either $a_{ij}\geq a_{min}$ or $\Tilde{a}_{ij}\geq a_{min}, \forall (j,i)\in\mathcal{E}$ in the case of the generic BAR model (\ref{eq: BAR_def2}). The data $\{\mathbf{x}(k)\}_{k=0}^{T}$ are generated according to the created ground truth models and then are used for network inference. We use the $F_1$ score, defined by
\begin{equation*}
    F_1 = \frac{2}{\text{recall}^{-1}+ \text{precision}^{-1}},
\end{equation*}
as the criterion to evaluate the performance of the algorithms. We note that \emph{recall} is the fraction of correctly recovered edges among the ground truth edges and \emph{precision} is the fraction of correctly recovered edges over all correct and incorrect edges identified by the algorithms.  More specifically,  an edge $(j,i)$ is viewed as being inferred if $\hat{a}_{ij}\geq ca_{min}$ or $\hat{\tilde{a}}_{ij}\geq ca_{min}$ for some empirically selected $c\in(0,1)$.

In Fig. \ref{fig: MLEvsBARobs_a}, we compare the performance of ML and closed-form estimators for the BAR model (\ref{eq: BAR_def}) with $p=10$ nodes and maximum in-degree $d_{max}=5$. In Fig. \ref{fig: MLEvsBARobs_c}, the performance of the ML, closed-form and BARobs estimators is demonstrated for a synthetic network with size $p=20$ and maximum in-degree $d_{max}=5$ corresponding to the generic BAR model (\ref{eq: BAR_def2}).  The $x$-axis corresponds to the number of observations $T$ and the $y$-axis to the $F_1$ score. In both cases, the ML estimator can achieve $F_1$ scores equal to $1$ for a sufficient sample size, e.g., $T=1200$ in Fig. \ref{fig: MLEvsBARobs_a} and $T=4000$ in Fig. \ref{fig: MLEvsBARobs_c}. In Fig. \ref{fig: MLEvsBARobs_c}, the BARobs algorithm also achieves an $F_1$ score equal to $1$ for a sufficiently large sample size. However, compared to the ML estimator, the BARobs requires a larger sample size. Finally, in Fig. \ref{fig: MLEvsBARobs_c}  both the ML and BARobs estimators outperform the closed-form estimator in (\ref{eq: LS_red2}).

Finally, focusing on the ML estimator which outperforms the other two estimators in terms of sample complexity, we present a real network experiment  in a biological application in which small sample sizes are critical \cite{SIGNET}. The underlying biological network consists of 43 nodes. The state of each node is binary and is updated according to a boolean rule defined by the states of some nodes in the network. We approximate the network by the generic BAR model (\ref{eq: BAR_def2}) and we generate pseudo-real data.  Fig. \ref{fig: Bio} illustrates the performance of the ML estimator on this data set.

\section{CONCLUSIONS}
\label{sec:concl}

In this paper, we studied the problem of estimating the parameters of a class of Markov chains  called BAR models. ML estimation for BAR chains was shown to be strongly consistent. Strong consistency was also established for closed-form estimators of these BAR models.






\section*{ACKNOWLEDGMENT}
This research has been supported in part by NSF Grants  NeTS 1718203,  CPS ECCS 1739189,  ECCS 16-09370,  CCF 1934986,  NSF/USDA Grant AG 2018-67007-28379,  ARO W911NF-19-1-0379, ECCS 2032321 and ONR Grant Navy N00014-19-1-2566.


\bibliographystyle{IEEEtran}
\bibliography{samples}

\end{document}